\documentclass[10pt]{amsart}
\usepackage{amsmath,amssymb}
\usepackage{color}
\usepackage{mathrsfs}
\usepackage{graphicx}

\numberwithin{equation}{section}
\theoremstyle{plain}
\newtheorem{theorem}{Theorem}[section]
\newtheorem{lemma}[theorem]{Lemma}
\newtheorem{corollary}[theorem]{Corollary}
\newtheorem{proposition}[theorem]{Proposition}
\newtheorem{obs}[theorem]{Observation}

\theoremstyle{definition}

\newtheorem{definition}{Definition}[section]

 {\theoremstyle{definition}\newtheorem{example}[theorem]{Example}}

\newcommand{\supp}{ \mbox{supp\;}}

\def\T{ \mathbb T}

\def\H{H^\infty}

\def\D{{ \mathbb D}}
\def\C{{ \mathbb C}}

\def\N{{ \mathbb N}}

\def\Re{{ \mathscr R}}
\def\e{\varepsilon}

\def\union{\cup}
\def\Union{\bigcup}
\def\inter{\cap}
\def\Inter{\bigcap }
\def\ov{\overline}

\def\ss{\subseteq}
\def\emp{\emptyset}

\def\buildrel#1_#2^#3{\mathrel{\mathop{\kern 0pt#1}\limits_{#2}^{#3}}}

\def\BP{Blaschke product}

\def\IBP{interpolating Blaschke product}

\def\ssi{\Longleftrightarrow}

\date{}

\begin{document}

\title [$t$-analyticity]{Partial regularity and $t$-analytic sets for Banach function algebras}


\author{Joel Feinstein}
\address{\small School of Mathematical Sciences\\
\small University of Nottingham\\
\small University Park\\
\small Nottingham~NG7~2RD, UK}

\email{Joel.Feinstein@nottingham.ac.uk}

 \author{Raymond Mortini}
\address{\small D\'{e}partement de Math\'{e}matiques\\
\small LMAM,  UMR 7122,
\small Universit\'{e} Paul Verlaine\\
\small Ile du Saulcy\\
\small F-57045 Metz, France}

\email{mortini@univ-metz.fr}

\subjclass{Primary 46J15 ; Secondary 30H50}
\keywords{Banach function algebras; regularity points; $t$-analytic sets; closed prime ideals;
bounded analytic functions; Douglas algebras; Alling's conjecture}

\begin{abstract}
In this note we introduce  the  notion of $t$-analytic sets. Using this concept, we  construct a class
of closed prime ideals in Banach function algebras and  discuss
some problems related to Alling's conjecture in $\H$.
A description of all closed $t$-analytic sets for the disk-algebra is given.
Moreover, we show that some of the assertions in \cite{dmz}
concerning the $O$-analyticity and $S$-regularity of certain Banach function algebras
 are not correct.
We also determine the largest set on which a Douglas algebra is pointwise regular.
\end{abstract}

 \maketitle


\section{Introduction}

In this note we will show that some refinements  of notions appearing
 in a paper by
Daoui, Mahzouli and Zerouali \cite{dmz} in connection with local/restricted decomposability
of multiplication operators on commutative, semisimple Banach algebras
 have a close connection with Alling's conjecture on the structure
of closed prime ideals in $\H$.

Let $X$ be a  compact  Hausdorff space. We denote by $C(X)$ the algebra of all
 continuous, complex-valued functions on $X$. With respect to the uniform norm, $\|\cdot\|_X$,
on $X$,    $C(X)$ is a unital, commutative, semisimple Banach algebra.

Let $A$ be a commutative, unital Banach algebra over the field of complex numbers $\C$. The  character space (or spectrum), $M(A)$, of $A$  is the set of all non-zero  multiplicative
linear functionals on $A$.  When endowed with the usual Gelfand topology (the relative weak-$*$-topology as a subset of $A^*$),
$M(A)$ becomes a compact Hausdorff space.
   Using Gelfand theory, we may identify $M(A)$ with the maximal ideal space of $A$. In the case where $A$ is semisimple, $A$ is isomorphic to the algebra $\hat A \subseteq C(M(A))$ consisting
   of the set of Gelfand transforms $\hat f: m\mapsto m(f)$  of elements $f\in A$.
In this case, identifying $f$ with $\hat f$, we may regard $A$ as a \emph{Banach function algebra} on $M(A)$.
With this identification, $\|f\|_{M(A)}$ is equal to the spectral radius of $f$ in the algebra $A$. If the spectral radius and the norm on $A$ are equal for all $f \in A$, then $A$ is a \emph{uniform algebra} on $M(A)$.

It is useful to remember the result that a commutative, unital Banach algebra $A$ is a uniform algebra if and only if
  $||f^2||=||f||^2$ for all $f\in A$.  Finally, the Shilov boundary, $\partial A$, of $A$
  is the smallest closed subset of $M(A)$ on  which every $f\in A$ achieves its maximum modulus.

  We refer the reader to the books of Dales \cite{Dales} (especially Chapter 4), Browder \cite{br},
  Gamelin \cite{gam},
  and Leibowitz \cite{lei} for various aspects of the theory of commutative Banach algebras.

Throughout this note, we will work only with semi-simple,   unital, commutative Banach algebras,
and  we look upon them  as Banach function algebras defined on their compact spectrum $M(A)$.

Some of our results are valid for more general algebras $A$ of complex-valued functions defined on sets other than $M(A)$, or for commutative Banach algebras which are not semisimple. However, these more general algebras are not our primary concern. For the sake  of clarity, we will not attempt to state our results in full generality.

\section{Partial regularity and $t$-analyticity: the definitions}

\subsection{Regularity of type I and type II}

First we recall the definitions of regularity and normality for Banach function algebras.
\begin{definition}
Let $A$ be a Banach function algebra. 
\begin{enumerate}
\item  $A$ is \emph{regular} if, for all $x\in M(A)$ and all
closed sets $E\subseteq M(A)\setminus\{x\}$, there exists $f\in A$ such that
 $f\equiv 1$ on $E$ and $f(x)= 0$;

 \item  $A$ is \emph{normal} if, for each pair of disjoint closed
 sets $E$ and $F$  in $M(A)$
 there exists $f\in A$  such that $f\equiv 1$ on $E$  and $f\equiv 0$ on $F$.
 \end{enumerate}
\end{definition}

 It is standard that regularity and normality are equivalent (see \cite[Proposition 4.1.18]{Dales}).
We next recall the definitions of the hull of an ideal and the kernel of a set, in order to define the hull-kernel topology on $M(A)$.

 \begin{definition}
 Let $A$ be a Banach function algebra. 
For $f\in A$, the \emph{zero set} of $f$ is the closed set $\{x\in M(A): f(x)=0\}$, and is denoted by $Z_A(f)$, or by $Z(f)$ if the algebra under consideration is clear.
Similarly, the \emph{hull} (or zero set)  of an ideal $J$ in $A$ is the closed set $\Inter_{f\in J} Z_A(f)$, and is denoted
by $Z_A(J)$, or $Z(J)$.

Let $E \subseteq M(A)$. Then the \emph{kernel} of $E$ is the closed ideal
\[
\{f\in A: f\equiv 0 \; \text{on}\; E\}\,,
\]
and  is denoted by $I_A(E)$ or $I(E)$.

Provided that there is no ambiguity in the algebra $A$ under consideration,
 we denote by $\hat E$ the hull of $I(E)$; that is $\hat E=Z(I(E))$.
The hull-kernel topology on $M(A)$ is  the topology whose closed sets are those sets
$E\subseteq M(A)$ such that $E=\hat E$.
\end{definition}

It is standard that the hull-kernel topology is a compact topology on $M(A)$ which is weaker than the Gelfand topology, and (hence) that these two topologies coincide if and only if the hull-kernel topology is Hausdorff. This occurs if and only if the Banach function algebra $A$ is regular (and hence normal). The hull-kernel topology and regularity are also closely related to the decomposability of multiplication operators on $A$, as discussed in \cite{n} (see also \cite{cf,fru,dmz,ln}).

For each ideal $J$ in $A$, we have that $Z(J)$ is, in fact, hull-kernel closed. Thus the hull-kernel
closed subsets of $M(A)$ are precisely the hulls of ideals in $A$. Moreover, for each $f\in A$, the hull of the principal ideal $fA$ is simply $Z(f)$, and so the zero sets of functions in $A$ are hull-kernel closed. In fact, the collection of \lq co-zero sets' $\{M(A)\setminus Z(f):f\in A\}$ is a base for the hull-kernel topology.

We now consider the identity map on $M(A)$, ${\rm Id}_{M(A)}$, regarded as a map from $M(A)$ with the hull-kernel topology to $M(A)$ with the Gelfand topology. It follows from the previous
paragraphs that
this map is continuous if and only if $A$ is regular.
In \cite{fs}, in order to discuss a localized (point by point) version of the connection between the hull-kernel topology and regularity, the points at which this map is continuous were investigated.

\begin{definition}

Let $A$ be a Banach function algebra. 
For each $f\in A$,
  $$\supp f:=\ov{\{x\in M(A): f(x)\not=0\}}$$
is the (closed) support of $f$.

\begin{enumerate}

\item
A point $x\in M(A)$ is called a {\it  regularity point of type one} for $A$, denoted by
 $x\in \Re_I$,  if for
every neighborhood $V$ of $x$ in $M(A)$, there exists
a function $f\in A$ such that $f(x)=1$  and the support of $f$ is contained in $V$.

\item
Let $X\subseteq M(A)$.
Then we say that  $A$ is {\it pointwise regular on $X$},  if  every  $x\in X$  is  a regularity point
of  type one for $A$.
We also say in that case  that $X$ is
{\it a set of regularity (of)  type one for $A$}.
\end{enumerate}
\end{definition}
It is clear that $A$ is regular if and only if $M(A)$ is a set of regularity  type one  for $A$.
 Later on, in Section \ref{three}, we will  determine for special algebras the largest set
  $X_{\max}$ in $M(A)$ on which $A$ is pointwise regular. That set may be empty, though.
 This is the case, for example, for the disk-algebra $A(\ov\D)$.
If $X$ contains the Shilov boundary of $A$ and if $A$ is pointwise regular on $X$,
then  $A|_X$  is a `regular function algebra on $X$' in the sense
as defined for example in \cite[p. 412]{Dales} or \cite[p. 189]{hof}.
We do not know whether the converse holds.

It follows from the results in  \cite{fs} that the points in
$\Re_I$ are exactly those points $x\in M(A)$ for which  the identity map
${\rm Id}_{M(A)}$ is continuous at $x$ in the sense described above \cite[Lemma 2.1]{fs}.
Hence we may say that $x$ is a regularity point of type one  if and only if $x$ is an
{\it $hk$-continuity point} for ${\rm Id}_{M(A)}$, (or briefly called a
 \lq point of continuity for $A$' in \cite{fs}).

It is also known that  $x\in M(A)$ is an $hk$-continuity point for ${\rm Id}_{M(A)}$ if and only if
every function in $A$ is continuous at $x$ when $M(A)$ is given the hull-kernel topology
(see \cite{fs}).
Of course, the set of all $hk$-continuity points for $A$ is the maximum $hk$-continuity set for $A$.
For completeness, let us say that $x\in M(A)$ is an {\it $hk$-discontinuity point},
if it is not an $hk$-continuity point.

Note also that if $S=M(A)\setminus X$, where $X\ss M(A)$,
then  our  notion of $X$ being a  set  of regularity type
one for $A$
is equivalent to the notion of $S$-regularity of $A$ given in \cite{dmz}.

Finally, we note that any regularity point of type one
(or equivalently  any $hk$-continuity point for $A$)
is necessarily contained in  the Shilov boundary $\partial A$ of $A$.
This fact, which is stated near the top of page 57 of \cite{fs}, is almost immediate from the definitions.

So this gives us one of many possible notions of
\lq partial regularity' for $A$.

We continue  this section by giving some related partial regularity conditions involving $k$-hulls, as introduced in \cite{gm1} in order to study spectral synthesis sets in $\H+C$; see also \cite{gm3}.

\begin{definition}
Let $A$ be a Banach function algebra.
\begin{enumerate}
\item
 For $x\in M(A)$, let $J_A(x)$ (or simply $J(x)$ if the algebra $A$ is unambiguous) be the ideal of those functions in $A$ that vanish
  (within $M(A)$) in a neighborhood of $x$.
  The  hull of  $J_A(x)$
   is denoted by $k_A(x)$ (or $k(x)$), and is called the $k$-hull of $x$.

\item   A point $x\in M(A)$ is called a {\it regularity point of  type two}, denoted by
   $x\in \Re_{II}$, if $k_A(x)=\{x\}$.
   \end{enumerate}
   \end{definition}

Note that  in  \cite{fs}, the elements   $x\in {\Re}_{II}$ were called $R$-points.

 For the readers convenience, we present a short proof that
  $A$ is regular if and only if each  $x\in M(A)$ is a  regularity point
  of  type two.

\begin{proposition}\label{khull}
 Let $A$ be a Banach function algebra.
 Then $A$ is  regular  if and only if $k_A(x)=\{x\}$ for every $x\in M(A)$;
 in other words ${\Re}_{I}=X\ssi {\Re}_{II}=X$.

\end{proposition}

\begin{proof}
Suppose that $A$ is  regular; that is ${\Re}_{I}=X$.   Let $x,y\in M(A)$ with $x\not=y$.
Since $y\in {\Re}_{I}$, there is $f\in J_x$ with $f(y)=1$. Thus $y\notin k_A(x)$.
Hence $k_A(x)=\{x\}$.

Now suppose that $k_A(x)=\{x\}$ for every $x\in M(A)$.  Let  $m\in M(A)$ and $E$
a compact subset of $M(A)\setminus\{m\}$.
 Then, for each $y\in E$, there exists $f_y\in J_A(y)$  with
$f_y(m)=1$.  A standard compactness argument then yields the desired
  function $f=f_{y_1}\dots f_{y_n}\in A$
with $f\equiv 0$ on $E$ and $f(m)=1$. Thus $m$ is a regularity point of type one and
 so $A$ is regular.
\end{proof}

The properties of $R$-points and $hk$-continuity points were discussed in detail in \cite{fs}. If $x$ is not an $R$-point for $A$, then the points of $k(x)\setminus\{x\}$ are points of
$hk$-discontinuity for $A$. Similarly,  $x$ is a point of $hk$-discontinuity for $A$ if and only if
there is at least one point $y \in M(A) \setminus\{x\}$ with $x \in k(y)$; each such point $y$ is then a non-$R$-point for $A$.
Another way to say this is that $x$ is a point of $hk$-continuity for $A$ if and only if there are no points $y \in M(A) \setminus\{x\}$ with
$x \in k(y)$. In general, though,  there are no inclusion relations between the set $\Re_I$ of
$hk$-continuity points and  the set
$\Re_{II}$ of $R$-points.

\subsection{$t$-analyticity}

We now define the notion of $t$-analyticity for subsets of $M(A)$, a notion which is in some sense
opposite to the notions of sets of regularity type one or two. It incorporates the fact that some trace of analyticity remains when the functions are restricted to a given set, essentially the identity principle/quasianalyticity for the restrictions of the functions.

\begin{definition}
Let $A$ be a Banach function algebra, and let $E\ss M(A)$.
We say that $E$ is a {\it $t$-analytic set} for $A$ (or simply \lq $E$ is \emph{$t$-analytic}' if the algebra $A$ is clear),
if for every $f\in A$ and every relatively open, nonvoid subset $U$ in $E$
one has that $f\equiv 0$
on $U$ implies that $f\equiv 0$ on $E$.
We say that $E$ is \emph{componentwise $t$-analytic} (for $A$) if every connected component of $E$ is a $t$-analytic set.
\end{definition}

The most obvious $t$-analytic sets are singletons and  the empty set.

 In \cite{dmz} the authors only consider the case where $E$ is an open set  in $M(A)$.
 They say that a Banach function algebra $A$ is \textit{$O$-analytic} if $O\ss M(A)$ is an open $t$-analytic set for $A$. They claim and attempt to prove that
 there exists a maximum open set $O$ for which $A$ is $O$-analytic. However this claim and its proof are incorrect.  We give the following easy example to demonstrate this. This example also shows that a union of open $t$-analytic sets need not be $t$-analytic, contrary to an implicit assumption made in \cite{dmz}

Let $X$ be a (non-empty) compact  set in $\C$ and $X^\circ$ its interior. Then $A(X)$ denotes the uniform algebra
$\{f \in C(X): f \text{ is analytic on } X^\circ\,\}$. By a theorem of Arens (see  \cite[Theorem 4.3.14]{Dales} or \cite[Chapter II, Theorem 1.9]{gam}), we have $M(A(X)) = X$.
Now consider
 the closed disks $O_1=\{z\in\C: |z|\leq 1\}$ and $O_2=\{z\in\C: |z-3|\leq 1\}$. Set $X=O_1 \union O_2$, and let $A$ be the uniform algebra $A(X)$.
Then $M(A)=X$, and the sets $O_1$ and $O_2$ are open $t$-analytic sets for $A$,
  but $O_1\union O_2$ is clearly
  not a $t$-analytic set for $A$. Moreover, no $t$-analytic set for $A$ can meet both $O_1$ and $O_2$. Thus there is no maximum open $t$-analytic set for $A$ in this case.

  All the subsequent \lq\lq results'' in \cite{dmz} that use this incorrect claim must  be
adapted appropriately.

Note, however, that it is true that, whenever $O_1$ and $O_2$ are open $t$-analytic sets for $A$ such that $O_1 \cap O_2 \neq \emptyset$, then
$O_1\union O_2$ is also an open $t$-analytic set for $A$.

\begin{lemma}\label{unions}
Let $A$ be a Banach function algebra and $C\ss M(A)$  a connected open set.
 Suppose that $C$ is a union of open sets $C_\lambda$ which  are $t$-analytic for $A$.
Then $C$ is $t$-analytic for $A$.
\end{lemma}
\begin{proof}
Let $U$ be an open set in $M(A)$ satisfying $U\inter C\not=\emp$.
Suppose  that $f\in A$ vanishes
 identically on $U\inter C$. We need to show that
$f\equiv 0$ on $C$.

By our assumption, there is some  $C_{\lambda_0}$  such that $U\inter  C_{\lambda_0}\not=\emp$.
Since $C_{\lambda_0}$ is $t$-analytic, $f\equiv 0$ on $C_{\lambda_0}$.
Now suppose to
the contrary that $f$ does not vanish identically on every $C_\mu$.
Let $\mathfrak C_1$ be the union  of all those $C_\mu$ such that $f\equiv 0$ on $C_\mu$ and let $\mathfrak C_2$ be the union of the remaining ones. Note that  by our hypothesis
$\mathfrak C_1$ and $\mathfrak C_2$ are nonvoid open sets.

Since $C$ is connected,
$\mathfrak C_1\inter \mathfrak C_2\not=\emp$. Now $f\equiv 0$ on the open set
$\mathfrak C_1$. Let $\mu_0$ be chosen  so that $C_{\mu_0}\ss \mathfrak C_2$ and
$C_{\mu_0}\inter \mathfrak C_1\not=\emp$.
The $t$-analyticity of $C_{\mu_0}$ now implies that $f\equiv 0$ on $C_{\mu_0}$.
This is a contradiction. Thus
 $\mathfrak C_2=\emp$ (since $C_{\lambda_0}\ss \mathfrak C_1\not=\emp$),
   and so $f\equiv 0$ on $C$.  Hence $C$ is $t$-analytic.
\end{proof}

\begin{proposition}
Let $A$ be a Banach function algebra, 
and let $\mathcal T$ be the set of all those open subsets $U$ of $M(A)$ such that all of the connected components of $U$ are also open in $M(A)$. Set
\[
\mathcal{S} = \{U \in \mathcal{T}: U \text{ is componentwise $t$-analytic for } A\,\}\,.
\]
Then $\bigcup_{U \in \mathcal S} U$ is also in $\mathcal S$, and is thus the maximum element of $\mathcal S$.
\end{proposition}

\begin{proof}
Let $V=\bigcup_{U \in \mathcal S} U$. We first show that $V\in \mathcal{T}$.
Obviously $V$ is open. Let $C$ be a connected component of $V$.  Write
$U=\Union_{\lambda\in \Lambda_U} C^U_\lambda$, where $C^U_\lambda$ is a connected
component of $U$.  Then
$$V=\Union_{U\in \mathcal S}\Union_{\lambda\in \Lambda_U} C^U_\lambda.$$
Since $V$ is a union of  connected open  sets,
it follows that the component $C$ of $V$  itself is the union of all those $C^U_\lambda$
 that are entirely contained in  $C$.
All these $C^U_\lambda$ are open; hence $C$ is open. Thus  $V\in \mathcal{T}$.

Next we show that $V\in \mathcal S$. Again, let $C$ be a connected component of $V$.
 In the representation for $C$ above, all the $C^U_\lambda$
are open and $t$-analytic for $A$; hence, by Lemma \ref{unions}, $C$ is $t$-analytic.
\end{proof}
In particular, if $M(A)$ is locally connected, then $\mathcal T$ is simply the whole of the topology on $M(A)$, and we see that the set $V$
 is the maximum, open, componentwise $t$-analytic set for $A$. However, if $M(A)$ is not locally connected, there may be a shortage of connected open sets.

For example, the spectrum of $\H$ is not locally connected. This follows from the fact
that for \IBP s $b$ with infinitely many zeros  and $\e>0$ small, the intersection of the level set
$\{x\in M(\H): |b(x)|<\e\}$ of $b$ with $\D$ is a disjoint union of infinitely many components
 (see \cite{ho}).

In general, when $E_1$ and $E_2$ are $t$-analytic sets for $A$ with $E_1 \cap E_2 \neq \emptyset$, it is not necessarily the case that $E_1 \cup E_2$ is a $t$-analytic set for $A$. Nor is it necessarily the case, even when $M(A)$ is locally connected, that there is a maximum, componentwise $t$-analytic set for $A$. To illustrate this, consider the following easy example, similar to our example above.

\begin{example}
\label{kissing_disks}
Consider the closed disks $X_1=\{z\in\C: |z|\leq 1\}$ and $X_2\break=\{z\in\C: |z-2|\leq 1\}$.
Set $X=X_1 \union X_2$, and let $A$ be the uniform algebra $A(X)$.
  Then $M(A)=X$, which is locally connected, and the sets $X_1$ and $X_2$ are (non-open) $t$-analytic sets for $A$ with
  $X_1 \cap X_2 = \{1\} \neq \emptyset$. However, $X=X_1\cup X_2$ is not $t$-analytic. For example, the function which is constantly zero on $X_2$ but equal to $z-1$ on $X_1$ is a function in $A$ which is identically zero on a non-empty open subset of $X$, but which is not constant on $X$.

  Since $X_1$, $X_2$ and $X$ are connected, it follows that there is no maximum, componentwise $t$-analytic set for $A$ in this case.
\end{example}

\subsection{Scheme of the paper}

In \cite{dmz} the existence of a smallest set $S_A$  for which $A$ is $S_A$-regular
is shown. In view of our earlier comments, $S_A$ is, in fact, equal to the set of all $hk$-discontinuity points for $A$, and so this is the complement of the maximum $hk$-continuity set for $A$.
It is claimed in \cite{dmz} that for Douglas algebras $D$, the set $S_D$ is equal to the union of the
support sets of all the representing measures. This is not correct. In fact, we will
 show in Section \ref{three} that $S_D$ equals $M(D)\setminus \partial D$,
 where $\partial D$ is the Shilov
  boundary of $D$.
 Also, their determination of the largest open set $O_D$ for which $D$  is $O_D$-analytic
 is not correct either. They appear to have been under the impression \ that the support sets
  (on $\partial D$) of the representing measures $\mu_m$
 for points $m\in M(D)$ have non-void interior. This is not correct.
 We will show in Section \ref{douglas} that for many, but not all,  Douglas algebras,
 including $\H+C$, there is no non-empty open $t$-analytic set $O\ss M(D)$.
 In this section  we will also  discuss the general structure of $t$-analytic sets in Douglas algebras,
  with emphasis on  the algebra $\H+C$. For certain points $x\in M(\H+C)$ we will explicitly determine the maximal $t$-analytic  sets containing $x$.

This will be preceded in  Section  \ref{analytic} by the proof of the fact
  that every $t$-analytic set for a Banach function algebra $A$
   is contained in a  maximal $t$-analytic set.
   These maximal $t$-analytic sets will automatically be closed.
   We will also show that, apart from singletons, the $t$-analytic sets for $A$
   do not contain any regularity points of type I or II.

  A complete characterization of the closed $t$-analytic sets for
  the disk-algebra is given in Section \ref{analytic}, too.

 This section also contains one of our main results: namely that  each
   $t$-analytic set induces a closed prime ideal in $A$.
    Alling's conjecture tells us that in $\H$ every non-maximal closed prime
  ideal coincides with the set of functions  vanishing identically on a nontrivial
  Gleason part. In the final section, Section \ref{alling}, we will discuss
   the relations between Gleason parts,
   the hulls of closed prime ideals  and the  $t$-analytic sets and
    address the question whether there are any counterexamples to Alling's conjecture.

\section{Partial regularity for Douglas algebras}\label{three}
As usual $\T$ denotes the unit circle $\{z\in\C: |z|=1\}$.
 Let $L^\infty(\T)$ be the space of (equivalence classes) of  measurable and
  essentially bounded
functions on $\T$, endowed with the essential supremum norm on $\T$.
Moreover, let $\H(\T)$ be the  subalgebra  of all those functions in $L^\infty(\T)$ whose
negative Fourier coefficients vanish.
 It is well known that, with respect to the essential supremum norm on $\T$,
$\H(\T)$ is isometrically isomorphic
to the uniform algebra $\H$ of all bounded analytic functions in $\D$.
A  closed subalgebra of $L^\infty(\T)$ that strictly contains $\H(\T)$ is called
a Douglas algebra.

The smallest Douglas algebra  is $\H(\T)+C(\T)$, which we denote for short by
$\H+C$. Here $C(\T)$ is the set of all complex-valued continuous functions on $\T$. The exact structure of these algebras is given in the celebrated
Chang-Marshall theorem (for a nice exposition see \cite{G}).  One of the consequences
is that the spectrum, $M(D)$, of  a Douglas algebra $D$ can be identified with a closed
subset of $M(\H+C)$.

Note that $M(\H+C)$ itself is $M(\H)\setminus\D$, which is sometimes called the corona
of the disk.

 Also, $M(D_1)=M(D_2)$ if and only if $D_1=D_2$.  Moreover,
the Shilov boundary, $\partial D$, of any Douglas algebra $D$  coincides with $M(L^\infty(\T))$.
In particular, $\partial D$ is a proper subset of $M(D)$ for any $D\not=L^\infty(\T)$.
 Whereas $L^\infty(\T)$, the largest Douglas algebra,
is a regular algebra (see also \cite{mo}), we have the following situation for  its proper subalgebras.
In what follows, we denote $k_{\H+C}(x)$ by $k(x)$, whenever $x\in M(\H+C)$.

\begin{theorem}
The Shilov boundary, $\partial D$, is the largest set on which a Douglas algebra
$D$ is pointwise regular.
\end{theorem}
This corrects statements in  the  paper \cite{dmz} by Daoui, Mahzouli and  Zerouali.

\begin{proof}
By the remarks that followed the definition  of regularity points, we first recall
that any regularity point of type one (if it exists) is necessarily contained in $\partial D$.
Next we show that $D$, actually, is pointwise regular on $\partial D$.  So
let $x\in \partial D$ and let $V$ be
an open neighborhood
of $x$ in $M(D)$. Let $y\in M(D)\setminus V$. Note that $\partial D=M(L^\infty)$.
Since by \cite{gm1,gm3} $k(y)\inter M(L^\infty)=\{y\}$ whenever $y\in M(L^\infty)$ and
$k(y)\inter M(L^\infty)=\emp$ whenever $y\in M(\H+C)\setminus M(L^\infty)$,
we see that there exists $f_y\in \H+C\ss D$ such that $f_y(x)=1$ and $f_y\equiv 0$ in a neighborhood
(within $M(\H+C)$) of $y$.  Since $ M(D)\setminus V$ is compact, we obtain a function
$f=f_{y_1}\dots f_{y_n}\in\H+C$ with $f(x)=1$ and $f\equiv 0$ on $M(D)\setminus V$.
Hence $D$ is pointwise regular on $\partial D$.  Thus $\partial D$ is
 the largest set on which $D$ is pointwise regular.
\end{proof}

As an immediate corollary we obtain

\begin{corollary}
Let $D$ be a Douglas algebra. Then the set of $hk$-continuity points for $D$ is equal to the Shilov boundary, $\partial D$.
\end{corollary}

In particular, if $D=\H+C$, then the set of points of regularity type I coincides with $\partial D$;
that is $\Re_I=\partial D$.
It is not known  whether  in $M(\H+C)$  there are points of regularity type II
(see \cite{gm1, gm3} and \cite{izu}.)
Note that this highly contrasts with the situation in  $\H$, where there are no regularity points
of type I and II.

\section{$t$-analytic sets}\label{analytic}

\subsection{General properties of $t$-analytic sets}

Throughout this section, $A$ is a Banach function algebra.
The following  result is immediate from the definition of a $t$-analytic set.

\begin{obs}
Let $E\ss M(A)$.
Then $E$ is a $t$-analytic set for $A$ if and only if
 $\overline E$ is  a $t$-analytic set for $A$.
\end{obs}

Recall that, for $E \subseteq M(A)$, $\hat E$ denotes the hull-kernel closure of $E$.
It is natural to ask whether the hull-kernel closure of a $t$-analytic set for $A$ is still a $t$-analytic set.
The following examples show that this is not necessarily the case.

\begin{figure}[h] 
\scalebox{.35} {\includegraphics{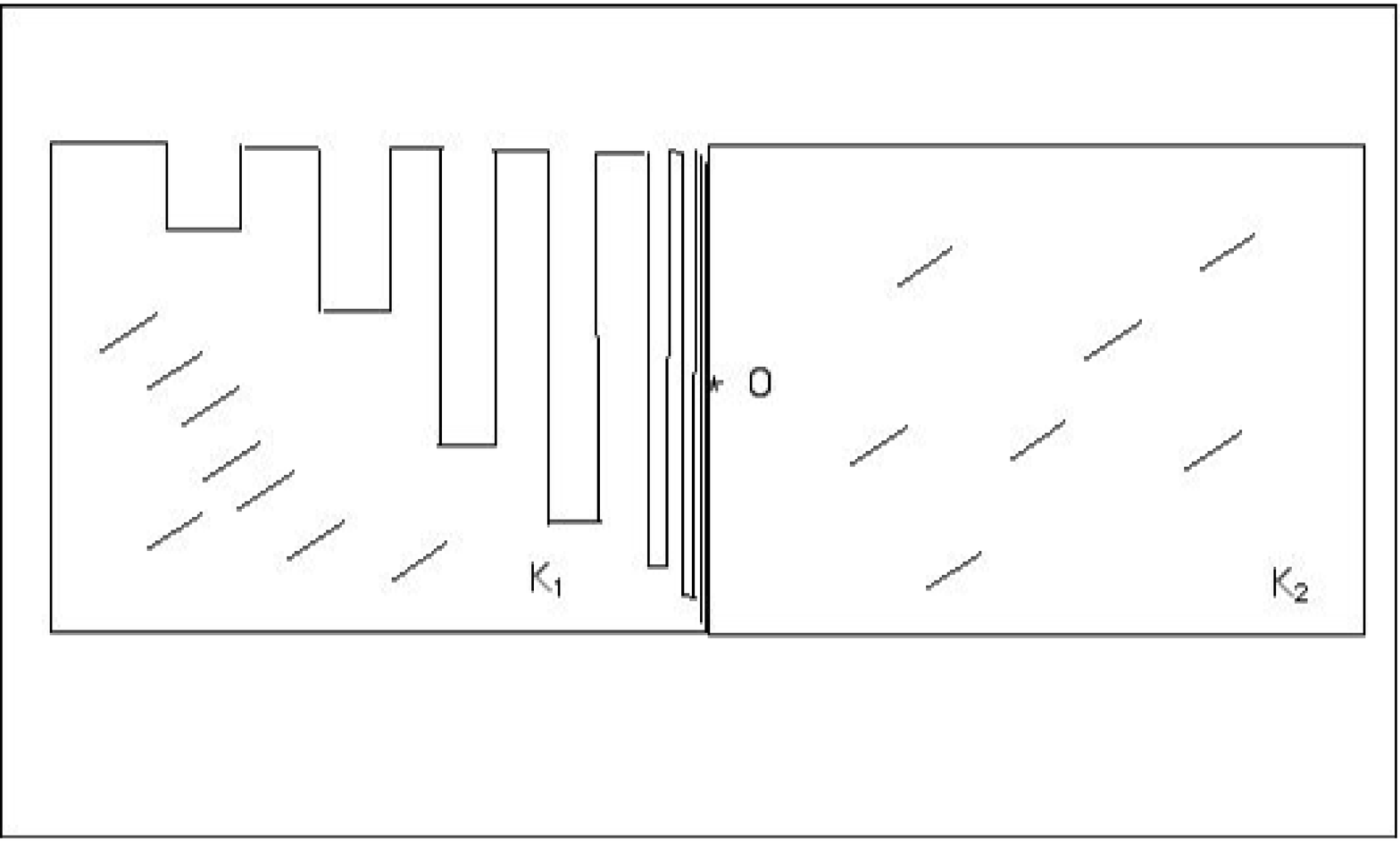}}
\caption{$\hat K_1=K_1\union K_2$}
\label{compact}
\end{figure}

\begin{example}\label{comb}
Let $K_j$ be the compact sets in figure 1 and set $K=K_1\union K_2$.\footnote{$K_1$ appeared
in a different setting in  \cite{ru}.}
It is well known \cite{pom}
that the Riemann map $f$ of $K_1^\circ $ onto the unit disk $\D$ admits a continuous extension
 to the boundary of $K_1$. This extension has constant value at the  whole boundary segment
 to the right  of $K_1$.
 Consider the function $g$ defined by
$g(z)=f(z)-f(0)$ if $z\in K_1$ and $g(z)=0$ if $z\in K_2$. Then $g\in A(K)$.
Now it is clear that $K_2$ and $K_1$ are maximal $t$-analytic sets; in particular,
$K$ is not $t$-analytic.  The hull of the ideal $I_{A(K)}(K_1)$ is $K$.
So $K=\hat K_1$ is the desired set.
\end{example}

The following alternative  example, intended  for adherents of the glueing technique,
 is unlike the first, not a point separating  algebra on its domain of definition $X$.
 Recall that $\N=\{0,1,2,\dots\}$ and $\N^*=\N\setminus\{0\}$.

\begin{example}
\label{glued}
Let $X_1$ and $X_2$ be two disjoint closed disks in the complex plane and set $X=X_1 \cup X_2$. Choose two convergent sequences of distinct points $a_n \in \partial X_1$ and $b_n \in {X_2}^\circ$ such that $b_0:=\lim_{n\to\infty} b_n \in {X_2}^\circ $.
Set $a_0 = \lim_{n\to\infty} a_n \in \partial X_1$.
Let \[
B=\{f \in A(X): f(a_n)=f(b_n) \text{ for all } n\in\N\,\}\,.
\]
Then $B$ is a closed sub-algebra  of $A(X)$  and $M(B)$ may be identified with the
quotient space $X/_\sim$ of $X$ obtained by identifying each point $a_n \in X_1$ with the corresponding point $b_n \in X_2$, $n\in\N$. This may also be thought of as glueing together the two disks $X_1$ and $X_2$ using these two sequences of points.

Let $\pi:X\to X/_\sim$  be the quotient mapping. By Fatou's interpolation theorem
\cite[p. 80]{hof} there exists
 $f\in A(X_1)$ that vanishes exactly at the points $\{a_n:n\in\N\}$.
Let $g=f$ on $X_1$ and $g=0$ on $X_2$. Then $g\in B$. Now as in example
\ref{comb}, we see that $\pi(X_j)$ are maximal $t$-analytic sets and the hull of the
ideal
$$\mbox{$I(\pi(X_1), X/_\sim):=\{h\in B: h\equiv 0$ on $X_1\}$}$$ equals $X/_\sim$.
\end{example}

It would be interesting to know for which algebras it is true that the hull-kernel closures of $t$-analytic sets are $t$-analytic.

\medskip

The following elementary result gives a useful connection between $t$-analytic sets and $k$-hulls of points.

\begin{proposition}\label{small}
Let $E\ss M(A)$ be a $t$-analytic  set for $A$, and  suppose that $x\in E$. Then $E\ss k_A(x)$.
\end{proposition}

\begin{proof}
By the definition of a $t$-analytic set, it is immediate that $J_A(x) \subseteq I_A(E)$, and hence that
$k_A(x) = Z(J_A(x)) \supseteq Z_A(I(E)) \supseteq E$.
\end{proof}

Our next result again shows the strong opposition between $t$-analyticity and regularity.

\begin{proposition}
\label{khull2}
Let $E\ss M(A)$ be a $t$-analytic set for $A$ such that $E$ has at least two points. Then $E\inter \Re_I=\emp$ and $E\inter \Re_{II}=\emp$.
In particular, if $A$ is regular, then the only non-empty, $t$-analytic sets for $A$ are singletons.
\end{proposition}
\begin{proof}

 Let $m \in E$. By Proposition \ref{small}, $E\ss k(m)$. Since $E$ has at least two points,
 $k(m)\neq \{m\}$. Thus  $m\notin \Re_{II}$. Hence $E\inter \Re_{II}=\emp$.

Now let $x,y\in E$ with $x\not=y$. Suppose, for contradiction, that $x\in \Re_I$. Let $V$
be a closed neighborhood of $x$ that does not contain $y$.  Then, by definition of the set $\Re_I$,
there exists $f\in A$ with $f(x)=1$ and ${\rm supp}\, f\ss V$. In particular, $f\in J(y)$ and
so $f$ vanishes on a non-empty open set that meets $E$. Since $E$ is  $t$-analytic, we
conclude that $f$ vanishes identically on $E$.  This contradicts the choice of $f$.
 Hence $E\inter \Re_I=\emp$.\\
The final part of the result now follows from, for example, the fact that $A$ is regular if and only if every point of $M(A)$ is a regularity point of type I (or II).
\end{proof}

 In \cite[Proposition 3.5]{dmz} it was claimed (in our terminology) that no open $t$-analytic set $O$ for $A$ can contain an $hk$-continuity point for $A$.
This claim is, in fact, false. For example, if $x$ is an isolated point of some compact space $X$, then $\{x\}$ is an open $t$-analytic set for $C(X)$, but (of course) $x$ is also an $hk$-continuity point for $C(X)$. The following observation corrects this statement, and is now a special case of part of the preceding proposition.

\begin{obs}\label{obser}
Let $O$ be an open $t$-analytic set for $A$ such that $O$ is not a singleton. Then $O$ does not contain any $hk$-continuity points
for $A$.
\end{obs}

\medskip

We now investigate the collection of $t$-analytic sets in more detail.
In what follows, the reader  should well distinguish between maximal $t$-analytic sets
and the maximum $t$-analytic set; the latter being non-existent, in general (see Example
\ref{kissing_disks}).

\begin{theorem}
Every $t$-analytic set for $A$ is contained in at least one maximal $t$-analytic set
for $A$. Moreover, the maximal $t$-analytic sets for $A$ are closed.
\end{theorem}

\begin{proof}
Let $E$ be a $t$-analytic set for $A$. Consider the (non-empty) collection $\mathcal F$
of all $t$-analytic sets for $A$ which contain $E$, partially ordered by inclusion.
Let $\{Y_\alpha: \alpha\in \Lambda\}$ be a non-empty chain of sets in $\mathcal F$.
Then this chain has an upper bound in $\mathcal F$; indeed
let
$$Y = \bigcup_\alpha Y_\alpha.$$
We show that $Y$ is a $t$-analytic set for $A$.

Let $f \in A$ and suppose that $f$ vanishes identically on a non-empty,  relatively open  subset, $U$,
of $Y$.  Choose $x_0 \in U$. Then for all large enough $Y_\alpha$,
we have $x_0 \in Y_\alpha$. For these $Y_\alpha$ we see that $f$
vanishes identically on the non-empty,  relatively open  subset
$U \cap Y_\alpha$ of $Y_\alpha$, and hence $f$ is constantly $0$ on
$Y_\alpha$. It now follows that $f$ vanishes identically on $Y$, as
required.

The existence of the required maximal $t$-analytic sets now follows from Zorn's Lemma.
Since the closure of such a maximal $t$-analytic set is still a $t$-analytic set, it is immediate that
such maximal $t$-analytic sets must be closed.
\end{proof}

It is natural to ask whether
each point $x\in M(A)$ is contained in a {\it unique} maximal
$t$-analytic set. However, this is not generally the case, as is shown by Example \ref{kissing_disks}. In that example, both of the sets $X_1$ and $X_2$ are maximal $t$-analytic sets containing the point $1$.
\medskip

A description of all $t$-analytic sets appears to be difficult in general.
At present, we know such a description only for general regular algebras (Proposition
\ref{khull2}),  and the disk-algebra (Theorem \ref{total}).
Part of the problem is that unions and subsets of $t$-analytic sets need not be $t$-analytic. There is more hope of describing the maximal $t$-analytic sets: we will give some partial results of this type
 (Theorem \ref{thin}).

\begin{proposition}\label{union}
Let $E$  be a closed set in $M(A)$ and suppose that $H$ is hull-kernel closed.
Then $E\union H$ is $t$-analytic if and only if $H\ss E$ and $E$ is $t$-analytic
 or if $E\ss H$ and $H$ is  $t$-analytic.
\end{proposition}

\begin{proof}
The \lq\lq if''direction is clear. Now suppose that $E\setminus H\not=\emp$ and
$H\setminus E\not=\emp$. Let $y\in E\setminus H$. There exists a function $f\in A$
such that $f\equiv 0$ on $H$, and $f(y)\not=0$.  But $f$ vanishes on the non-empty,
relatively open
set $(E\union H)\setminus  E=H\setminus E $ in $E\union H$.
Hence $E\union H$ cannot be $t$-analytic.
\end{proof}

\begin{obs}\label{prelim}
Suppose that $E_j$ are closed sets in $M(A)$ and that $E_1\union E_2$ is $t$-analytic.
If $E_1\not\ss E_2$, then $E_2\ss\hat E_1$.
\end{obs}
\begin{proof}
We show that $E_1\union E_2\ss \hat E_1$.
Let $f\in A$ vanish identically on $E_1$. Since
$$U:=(E_1\union E_2)\setminus E_2=E_1\setminus E_2\not=\emp,$$
$f$ vanishes on the non-empty,  relatively open subset $U$ of $E_1\union E_2$. Hence, the $t$-analyticity
implies that $f$ vanishes identically on $E_1\union E_2$. Thus $E_1\union E_2\ss \hat E_1$,
by definition of $\hat E_1$.
\end{proof}

\begin{corollary}
Suppose that $E_j$ are closed sets in $M(A)$ and that $E:=E_1\union E_2$ is $t$-analytic.
If neither $E_1$ is contained in $E_2$ nor $E_2$ contained in $E_1$, then
$\hat E=\hat E_1=\hat E_2$.
\end{corollary}
\begin{proof}
By the observation \ref{prelim} above, $E_2\ss \hat E_1$ and $E_1\ss \hat E_2$.
Since the hull-kernel closure of  a finite  union is the union of the hull-kernel closures,
 we see that $\hat E=\hat E_1=\hat E_2$.
\end{proof}

Using Proposition \ref{union}, we can give  certain necessary conditions for a closed set to
be $t$-analytic  in a Banach function algebra.

\begin{definition}
Let $E$ be a closed subset of  a compact Hausdorff space. A non-void, proper subset
 $S$ of $E$ is said to be {\it isolated} \footnote{~~or relatively clopen}
 (within $E$),  if there exist two disjoint open sets $U$ and $V$ such that
$$\mbox{$S\ss U$ and $E\setminus S\ss V$};$$
that is if $S$ can be separated from its complement within $E$.
\end{definition}

\begin{corollary}\label{cara}
Let $E\ss M(A)$ be a closed  $t$-analytic set for $A$.
Then either  $E$ is  a singleton or  $E$ does not contain
 any isolated hull-kernel closed subsets.
\end{corollary}

\begin{proof}
This follows immediately from Proposition \ref{union}.
\end{proof}

The following result  is a  slight generalization  of the previous corollary \ref{cara}.
\begin{proposition}\label{big-small}
Let $E\ss M(A)$ be a closed $t$-analytic set for $A$. Then for every hull-kernel closed set $H$
in $M(A)$ either $E\setminus H=\emp$ or $E\setminus H$ is dense in $E$.
\end{proposition}
\begin{proof}
Suppose, for contradiction,  that  $\emp\not= \ov{E\setminus H} \subsetneqq E$. Then there
is $y\in E\setminus H$ and $f\in A$ such that $f\equiv 0$ on $H$, but $f(y)\not=0$.
In particular $f\equiv 0$ on the nonvoid (relatively) open set
$E\setminus \bigl(\ov{E\setminus H}\bigr)\ss E$.
Since $E$ was assumed to be $t$-analytic, we obtain the contradiction that $f(y)=0$.
\end{proof}

\subsection{$t$-analytic sets in $A(\ov\D)$.}

If $A$ is the disk-algebra, then we obtain a complete  characterization of the closed $t$-analytic sets.
Normalized Lebesgue measure on $\T$ will be denoted  by $\sigma$.

 \begin{theorem}\label{total}
 Let $E$ be closed subset in $\ov\D$. Suppose that $E$ has at least two points. Then $E$
  is $t$-analytic for $A(\ov\D)$ if and only if $E$ does not contain
 any isolated points  and  has the property that for every  open arc $I\ss\T$  that  does not meet
 $\ov{E\inter\D}$,
  one has that either $I\inter E=\emp$ or $\sigma(I\inter E)>0$.
 \end{theorem}
 \begin{proof}
 The necessity (via contraposition) of the condition follows from Corollary \ref{cara} by using the well known facts ( see \cite{hof})
 that a closed set $S$  in $\ov\D$ is hull-kernel closed if and only if $S$ is the zero set
 of a function in $A(\ov\D)$  and that these zero sets are either  unions of Blaschke sequences with
 closed subsets of $\T$ of Lebesgue measure zero or the entire closed disk $\ov\D$.

 To prove the converse, we first note that if $E\inter \D\not=\emp$, then, by our hypotheses,
  any open set $U$  that meets $E$ within $\D$ has the property that  $U\inter E$ is uncountable.
  Thus, whenever
 $f$ vanishes on the nonvoid set  $U\inter E\inter \D$, $f$ vanishes everywhere on $\ov \D$.

 Now if $U$ is an open set in $M(A(\ov\D))$ that meets $E\inter \T$, then $U\inter \T$ is a union
 of pairwise disjoint open arcs $I_j$. Let $f\equiv 0$ on $U\inter E$.
  Suppose that  $I_j\inter E\not=\emp$ and that
$f\equiv 0$ on $I_j\inter E$. If $I_j\inter E\ss \ov{E\inter\D}$, then $U\inter E\inter \D\not=\emp$,
and so, by the previous paragraph, $f\equiv 0$ on $\ov \D$. If $I_j\inter E$ is not entirely
contained in  $\ov{E\inter\D}$, then there exists a subarc $J\ss I_j$ such that
$J\inter \ov{E\inter\D}=\emp$ and $J\inter E\not=\emp$. In that case, the hypotheses
implies that $\sigma(J\inter E)>0$. But $f\equiv 0$ on $J\inter E$. Hence, in this case, too, $f\equiv 0$
on $\ov\D$.

To sum up, we have shown that if $f\equiv 0$ on $U\inter E$, then $f$ is the zero function whenever
$U\inter E\not=\emp$.  Hence, $E$ is $t$-analytic.
   \end{proof}

   We deduce the following corollaries.

 \begin{corollary}
 The only $t$-analytic hull-kernel closed sets for $A(\ov \D)$ are the empty set, singletons and
 the whole spectrum.
 \end{corollary}

  \begin{corollary}
  Every connected closed set in the spectrum of $A(\ov \D)$ is $t$-analytic.
 \end{corollary}
 \begin{proof}
 If $E\ss\ov\D$ is a singleton, then $E$ is $t$-analytic. So suppose that $E$ contains
 at least two points. Since $E$ is connected, $E\inter \D$ is either empty, or contains no
 isolated point.  Now if $I\ss\T$ is an open arc with $I\inter \ov{E\inter \D}=\emp$, then
 either $I\inter E=\emp$ or the connectedness of $E$ implies that $I\inter E$ contains
 a closed arc. Hence $\sigma(I\inter E)>0$. By Theorem \ref{total}, $E$ is $t$-analytic
 for $A(\ov\D)$.
 \end{proof}

  \begin{corollary}
  Every closed set  in the spectrum of $A(\ov \D)$ is componentwise  $t$-analytic.
 \end{corollary}

 Recall that a {\it uniqueness set} for  a Banach function algebra $A$ is a
 nonvoid set $E\ss M(A)$ such that
 any two functions $f$ and $g$ in $A$ that are equal on $E$ already coincide on $M(A)$;
 in other words if $f|_E=g|_E$ implies  $f=g$.

Hence we can conclude from  Theorem \ref{total}  that any non-empty $t$-analytic set $E$ for $A(\ov\D)$ is  either  a singleton or a uniqueness set for $A(\ov\D)$.

 Such a result is not valid in general; for example in $\H$ closures of non-trivial Gleason parts
 outside $\D$ are $t$-analytic sets (see Section \ref{douglas}), but of course not
 uniqueness sets. On the other hand, the class of uniqueness sets in $A(\ov\D)$ is much larger
 than the class of $t$-analytic sets; as  examples we may take the set
 $\{1-\frac{1}{n}:n\in \N^*\} $ (because  $(1-\frac{1}{n})$ is not a Blaschke sequence),
 or just $\{\frac{1}{n}: n\in\N^*\}$.
   \medskip

 \subsection{$t$-analytic sets and prime ideals}

The main result of this section is the following surprising relation between $t$-analytic sets and
closed prime ideals.

\begin{theorem}\label{main}
Let $A$ be a Banach function algebra, and let $E\ss M(A)$ be a nonvoid
$t$-analytic subset for $A$. Then
$\mbox{$I_A(E)=\{f\in A:  f \equiv 0$ on $E\}$}$ is a closed prime ideal in $A$.
\end{theorem}

\begin{proof}
Clearly $I_A(E)$ is a closed ideal.
Let $f$ and $g$ be in $A$ with $fg \in I_A(E)$.
We show that at least one of $f$ and $g$ is in $I_A(E)$.
Suppose that $f$ is not in $I_A(E)$. Then $U:=E\setminus Z_A(f)$ is a
 nonvoid, relatively open subset
of $E$. Since $E \subseteq Z_A(f) \cup Z_A(g)$, it follows that
$g$ is constantly $0$ on $U$ and hence, since $E$ is $t$-analytic,
 $g \in I_A(E)$.
\end{proof}

\begin{corollary}\label{conn}
Suppose that $E$ is a hull-kernel closed $t$-analytic set for $A$. Then $E$ is connected.
\end{corollary}
\begin{proof}
Since the hull of the closed  ideal $I_A(E)$ equals $\hat E=E$, we see that
the quotient algebra $A/I_A(E)$ is a Banach algebra with spectrum $E$.
The Shilov idempotent theorem and the fact that $I_A(E)$ is prime now imply that
$E$ must be connected.
\end{proof}


\section{$t$-analyticity in Douglas algebras}\label{douglas}

In this section we study the structure of the $t$-analytic sets in Douglas algebras with emphasis on
$\H+C$.
Contrary to the claims in \cite{dmz}, we will see that there is no (non-void) open $t$-analytic set for the smallest Douglas algebra, $\H+C$.
One of the reasons is that
 in $\H+C$ the $k$-hulls of points are very small sets. To show this, we need a little detour.
 Since at some points of the proofs, we also need the notion of interpolating sequences
for Banach function algebras $A$, we recall that notion here.

A sequence $(x_n)$ of points in $M(A)$  is called an {\it $A$-interpolating sequence}, if
for every bounded sequence, $(a_n)$, of complex numbers there exists $f\in A$ such that
$f(x_n)=a_n$ for all $n$. In the case of the algebra $\H$, the interpolating
sequences contained in $\D$  were characterized by L. Carleson
 \cite{ca1, ca2}. These are the sequences $(z_n)$ satisfying the condition
 $$ \inf_{k\in \N}{\prod_{j:j\not=k} \left| \frac{z_j-z_k}{1-\overline{ z}_j z_k}
 \right|} \, \,  >0.$$

 If  additionally $\lim_{k\to \infty}{\prod_{j:j\not=k} \left| \frac{z_j-z_k}{1-\overline{ z}_j z_k}
 \right | } \, \,=1$, then one says that $(z_n)$ is a {\it thin sequence}. The Blaschke product
 associated with a thin sequence is called a {\it  thin \BP}.
 It is known that any sequence in $\D$ converging to the boundary has a thin subsequence
 (see \cite{gomo}).

 Recall that, by Carleson's Corona theorem, $\D$ is dense in $M(\H)$.
 According to Hoffman \cite{ho}, the set of points $x\in M(\H+C)$ that belong to the closure of an $\H$-interpolating sequence in $\D$ is denoted by $G$. The points in $M(\H+C)$ belonging
  to the closure of thin sequences are called {\it thin points}. Both $G$ and the set of thin points
  are dense in $M(\H+C)$.

Let $QC=\{f\in\H+C: \ov f\in\H+C\}$ be the largest $C^*$-subalgebra  of $\H+C$.
Let $\rho:  M(\H+C)\to M(QC)$ be the restriction mapping $m\mapsto m|_{QC}$.
It is well known that $\rho$ is a surjection. For each $x\in M(QC)$, let $\rho^{-1}(x)$
be the QC-level set associated with $x$. These $QC$-level sets form a partition
of $M(\H+C)$ into compact sets. Thus each $m\in M(\H+C)$ is contained in a unique
$QC$-level set.  Moreover, a function  $f\in \H+C$ belongs to $QC$ if and only if $f$
is constant on the $QC$-level sets.

 The following lemma is surely known to specialists in the field, but we do not know
an explicit reference. Our proof depends on
Hoffman's theory of the structure of the maximal ideal space of $\H$ and on Wolff's
interpolation theory for $QC\inter \H$ functions (see \cite{sw} and \cite{wo1}).

\begin{lemma}\label{qc}
The $QC$-level sets are nowhere dense in $M(\H+C)$.
\end{lemma}

\begin{proof}
Suppose, for contradiction,
 that some $QC$-level set has nonempty interior $U$ in $M(\H+C)$. Let $x\in U$.
We first prove (by a purely topological reasoning) the existence of an open
set  $O$ in $M(\H)$ such that
$x\in \ov O\inter M(\H+C)\ss U$. To this end,
let $V$ be an open set in $M(\H+C)$ such that $x\in V\ss\ov V\ss U$.
Since $M(\H)$ is a normal space, we can separate the compact sets $M(\H+C)\setminus U$
and $\ov V$ by two open sets in $M(\H)$. Let us call the  open set  containing $\ov V$ by $O$.
 It is now clear that $x\in \ov O\inter M(\H+C)\ss U$.

 Since  $G$ is  dense in $M(\H+C)$, there is $y\in G\inter (O\inter M(\H+C))=G\inter O$.
 By definition of $G$,  there exists an $\H$-interpolating   sequence $(z_n)$
   in $O\inter\D$  such that $y\in \ov{\{z_n:n\in\N\}}$.  
    Thus
   $$\ov{\{z_n:n\in\N\}}\inter M(\H+C)\ss \ov O\inter M(\H+C)\ss U.$$
 Let $(w_n)$ be a thin subsequence of $(z_n)$. Then by \cite{wo1}, $(w_n)$
 is an interpolating sequence for functions  in $QC\inter \H$. In particular,
 $(w_n)$ has at most a single cluster point in each $QC$-level set. Since the set of
 cluster points is infinite, we get a contradiction to our assumption that $U$ is contained
 in a single $QC$-level set.
 \end{proof}

\begin{lemma}\label{dense}
The $k$-hulls $k(x)$ of points in $M(\H+C)$ are nowhere dense.
\end{lemma}
\begin{proof}
Since $QC$ is isometrically isomorphic to $C(M(QC))$, $QC$ is a regular
subalgebra of $\H+C$.  By Proposition \ref{khull} the $k$-hulls of points in $M(QC)$
are singletons. Now it is clear that for any $m\in M(\H+C)$ the $k$-hull
 $k(m)$ is contained in the unique
$QC$-level set  $E_m$ containing $m$. Since, by Lemma \ref{qc}, $E_m$ is nowhere dense in $M(\H+C)$,
the same is true for its subsets; in particular for $k(m)$.
\end{proof}

\begin{theorem}
There is no non-empty, open $t$-analytic set  in $M(\H+C)$.
\end{theorem}

\begin{proof}
Suppose that  $O$ is a $t$-analytic open subset of $M(\H+C)$.
Let $x\in O$.
By Proposition \ref{small}, $O\ss k(x)$.   Since  by Lemma \ref{dense},
$k(x)$ is nowhere dense, we see that $k(x)$ has no interior points. Thus  $O=\emp$.
\end{proof}

On the other hand, there do exist Douglas algebras that are $O$-analytic for some
open sets. To present an example, we need some further facts and notions from
Hoffman's theory.

 One of those facts tells us that every $m\in M(\H)$ has a unique representing measure
on the Shilov boundary of $\H$. The associated support of that measure is denoted by
 $\supp m$. Note that the Shilov boundary of $\H$ is identified with $M(L^\infty)$.

Recall that a Gleason part of a uniform algebra  is a set of the form
 $$P(m)=\{x\in M(A): \rho(x,m)<1\},$$
where $\rho(x,m)=\sup\{|\hat f(x)|: ||\hat f||_\infty \leq 1,  \hat f(m)=0\}$
is the pseudo-hyperbolic distance.
The Gleason parts for $\H$, and hence for any Douglas algebra, were characterized
by Hoffman (\cite{ho}).  It turned out that they can be divided into two classes:
the  trivial parts (these are those for which $P(m)$ is the singleton $\{m\}$,)
 and the nontrivial parts. Hoffman showed that $P(m)$ is non-trivial if and only if $m$
 belongs to the closure of an $\H$-interpolating sequence, that is $m\in G$.
 For example, every point of the Shilov boundary of $\H$ has a trivial Gleason part.
  Also, for every $m\in M(\H)$
 there exists a continuous  map $L_m$ of $\D$ onto $P(m)$ such that  $f\circ L_m$
 is analytic for all $f\in \H$. Here $L_m$ is a bijection if and only if $P(m)$
 is non trivial, and $L_m$ is constant otherwise.  Due to the Chang-Marshall theory,
 it is known that the Gleason parts $P(x)$  for Douglas algebras  coincide  with
 those for $\H$ whenever $x\in M(D)$.

 The next proposition follows immediately from the above facts concerning analytic structure in the character space.

 \begin{proposition}\label{exo}
 Let $D$ be a Douglas algebra (or $D=\H$) and let $m\in M(D)$. Then $P(m)$ and $\ov{P(m)}$
 are $t$-analytic sets.
 \end{proposition}

  The following follows from standard arguments in the theory of Douglas algebras
 (see \cite{G} and \cite{gam}).

 \begin{example}

Let $m$ be a thin point in $M(\H+C)$, and consider the Douglas algebra
\[
D=\{f\in L^\infty: f|_{\text {\supp} m}\in \H|_{\text{\supp} m}\}\,.
\]

Then  $M(D)=M(L^\infty) \union \{x\in M(\H): \supp x\ss \supp m\}$.
By \cite{bu}, the thin \BP\ $b$ associated with $m$ vanishes  only at $m$, when looked upon
as an element in $D$. Moreover $O:=P(m)=\{x\in M(D): |b(x)|<1\}$ is open in $M(D)$.
It is evident that $O$ is $t$-analytic.
\end{example}

In view of Proposition \ref{exo} and the fact that we were unable to find other examples,
 we ask the following question.

\medskip

{\bf Q1} Let $D$ be a Douglas algebra.
 Is every maximal $t$-analytic set for $D$ the closure of a Gleason part?

\medskip
In the following we give some results supporting this conjecture. It is also closely related to
Alling's conjecture (see Section \ref{alling}).

Note that  if $E\ss M(D)$ is a $t$-analytic set for $D$, then it is
$t$-analytic for $\H+C$.

\begin{lemma}\label{kd}
Let $D$ be a Douglas algebra and $x\in M(D)$. Then $k_D(x)\ss k_{\H+C}(x)$.
\end{lemma}
\begin{proof}
This follows from the definition that $k_D(x)$ is the intersection of the zero sets
of all functions  in $D$ that vanish in neighborhood (within $M(D)$) of $x$ and
the fact that $\H+C\ss D$.
\end{proof}

In what follows we  use again the symbol $k(x)$ for $k_{\H+C}(x)$.

\begin{proposition}\label{shilov}
Let $x$ be  a point in the Shilov boundary $\partial D$ of a Douglas algebra $D$.
Then $\{x\}$ is a maximal $t$-analytic set for $D$ .
\end{proposition}
\begin{proof}

Let $E$ be a $t$-analytic set for $D$  containing $x$.
By Lemma \ref{small}, we have that  $E\ss k_D(x)$.
By \cite{gm1,gm3} $k(x)\inter M(L^\infty)=\{x\}$; hence we see from Lemma \ref{kd} above
that $k_D(x)\inter \partial D=\{x\}$. Now if $y\in E\setminus \partial D$, then, by \cite{gm1}, $k(y)\inter M(L^\infty)=\emp.$ Thus there exists $f\in \H+C\ss D$  that vanishes in a neighborhood
(within $M(\H+C)$) of $y$, but not at $x$.  This contradicts the $t$-analyticity of $E$.
Thus $E=\{x\}$.
\end{proof}

\begin{proposition}

Let $E$ be a non-empty,  hull-kernel closed subset of the set $X$ of trivial points in $M(\H+C)$.
Then $E$ is $t$-analytic if and only if $E$ is a singleton.
\end{proposition}

\begin{proof}
The ``if" part is trivial. Now suppose that $E$ is $t$-analytic.
By Corollary \ref{conn}, $E$ must be connected. However, by \cite{sua},
the set $X$ is totally disconnected. Hence $E$ is a singleton.
\end{proof}

\begin{obs}\label{twopoint}
Let $E\ss M(\H+C)$ be a  $t$-analytic set for $\H+C$.  Suppose that
$x$ and $y$ are two points in $E$. Then $k(x)=k(y)$.
\end{obs}
\begin{proof}
By Lemma \ref{small}, $E\ss k(x)\inter k(y)$. Suppose, for contradiction,  that $x\notin k(y)$.
Then there exist $f\in J(y)$ such that $f(x)\not=0$. Since
$f$ is non-constant on $E$,  but $f$ vanishes identically on a non-empty,
relatively open set in $E$ we see that $E$ cannot be
a $t$-analytic subset of $M(\H+C)$. We conclude that $x\in k(y)$ and so, by \cite{gm3},
$k(x)\ss k(y)$. Interchanging the role of $x$ and $y$ yields that $k(x)=k(y)$.
\end{proof}

In analogy to the observation \ref{twopoint} above,  we have  a similar situation
whenever $E$ is the hull  of a closed prime ideal  (see also question Q3 in the next section).

\begin{obs}\label{primehull}
Let $I$ be a closed prime ideal in $\H+C$. Then $k(x)=k(y)$ whenever $x,y\in Z(I)$.
\end{obs}
\begin{proof}
By \cite[Theorem 1.9]{gm2},  $Z(I)\ss k(x)\inter k(y)$. Suppose, for contradiction,
 that $y\notin k(x)$.
Then there exists $f\in J(x)$ with $f(y)\not=0$. But by \cite[Corollary 1.11]{gm2} $J(x)\ss I$;
thus $f\in I$ and so, since $y\in Z(I)$,  $f(y)=0$. This is a contradiction.  Hence $y\in k(x)$.
Thus, by \cite{gm3}, $k(y)\ss k(x)$. Since the argument is symmetric in $x$ and $y$,
we obtain the assertion that $k(x)=k(y)$.
\end{proof}

We can now strengthen the assertion of the observation \ref{twopoint}.

\begin{corollary}
Let $E$ be a $t$-analytic set for $\H+C$. Then $k(x)=k(y)$ for each  $x,y\in \hat E$.
\end{corollary}
\begin{proof}
This follows from Theorem \ref{main} that $I_{\H+C}(E)$ is a prime ideal and observation \ref{primehull}.
\end{proof}

We note that, for general function algebras, this assertion is no longer true. In fact,
let $A=A(K)$ be the example \ref{comb}.  Then $k(x)=K_1\union K_2$ for every
$x\in K_1$, where $K_1$ is a maximal $t$-analytic set,
but $k(y)=K_2$ for $y\in K_2^\circ$. Although $y$ belongs to
the hull of $I_A(K_1)$, $k(y)$ is strictly smaller than $k(x)$.
\medskip

We now come to the main (and concluding) result of this section.
In what follows, let $\{|f|<1\}=\{x\in M(\H+C): |f(x)|<1\}$
whenever $f\in \H+C$.

\begin{theorem}\label{thin}
Let $E$ be a maximal $t$-analytic set for $\H+C$ containing a thin point $x\in M(\H+C)$.
Then $E=\ov{P(x)}$.
\end{theorem}
\begin{proof}
By Lemma \ref{small}, we know that $E\ss k(x)$. Let $b$ be a thin \BP\ with $b(x)=0$.
Let  $Q(x)$ be the $QC$-level set containing $x$. We know  that $k(x)\ss Q(x)$
 (see the proof of Lemma \ref{dense} or \cite{gm1})  and that   $\{|b|<1\} \inter Q(x)=P(x)$ (see \cite{izu}). Thus
 $\{|b|<1\}\inter k(x)=P(x)$.  Moreover,  $\ov{P(x)}$ is hull-kernel closed (see \cite{go}).
  Two cases now appear.

 {\it Case 1} If $k(x)\ss \ov{P(x)}$, then it easily follows that $k(x)= \ov{P(x)}$. Hence,
  by the maximality of $E$ and Proposition \ref{exo},  $E=\ov{P(x)}$.
 (Note this case can actually happen, see \cite{gm2}).

 {\it Case 2} If $k(x)\not\ss \ov{P(x)}$,
 let $y\in k(x)\setminus \ov{P(x)}.$
  Then there exists $f\in \H$ such that $f\equiv 0$ on $\ov{P(x)}$, but $f(y)\not=0$.
 Thus the zero set of $f$ meets $E$ in a relatively open set, namely $\{|b|<1\}\inter E$,
 but $f(y)\not=0$. Since $E$ is assumed to be $t$-analytic, we conclude that  $y\notin E$.
 Thus $E\ss \ov{P(x)}$. Proposition \ref{exo} and the maximality of $E$ now imply that $E= \ov{P(x)}$.
 \end{proof}

 We observe that
 by \cite{izu}, $k(x)$ is a minimal $k$-hull for every $x\in Z(b)$ whenever $b$ is a thin \BP\
 and that $P(x)$ is a maximal Gleason part with maximal closure (see \cite{bu}).

\section{Is there a counter-example to Alling's conjecture?}\label{alling}

Alling's famous conjecture \cite{all} reads as follows: Let $I$ be a non-maximal closed prime ideal
in $\H$.  Does there exist a point $x\in G$ such that
$$I=\{f\in\H: f\equiv 0\;{\rm on}\; \ov{P(x)}\}?$$
(See also \cite{mor}).
It is easy to see that every such ideal actually is prime. In \cite{gm0}, it was shown
that if $I\not=(0)$ is a non-maximal closed prime ideal in $\H$ then $I=I_{\H}(E)$ for some  hull-kernel closed
subset $E\ss M(\H+C)$. The analogous result is true for prime ideals in $\H+C$; see \cite{gm2}.
So Alling's conjecture reduces to the problem of determining the hull of $I$:
Is $Z_A(I)=\ov{P(x)}$ for every non-maximal closed prime ideal $I$ in $A=\H$ or $\H+C$?

Theorem \ref{main} gives us   another way to construct closed prime ideals.
It says that if $E$ is a $t$-analytic set for $A$, then  $I_A(E)$ is prime.

Thus every hull-kernel closed  $t$-analytic set $E$ for $\H+C$ that is not  a singleton
yields a non maximal closed prime ideal.  This raises the following question:
\medskip

{\bf Q2}
Let $E\ss M(\H+C)$ be a hull-kernel closed $t$-analytic set for $\H+C$. Suppose that
$E$ is not a singleton.
Is $E$ equal to $\ov{P(x)}$ for some $x\in G$?

\medskip

Unfortunately Theorem 5.12 does not prove Alling's conjecture
for prime ideals containing a thin point in their hull since we do not know whether
the hull of such an ideal is a $t$-analytic set.
This raises the next question.

\medskip

{\bf Q3}   Is the hull of every closed prime ideal in  $\H+C$ a $t$-analytic set?\medskip

Example \ref{comb} shows that it is not always true that the hull of a closed prime ideal
in an arbitrary Banach function algebra
is a $t$-analytic set. In that example the zero ideal is prime, and the hull of the zero ideal is $M(A)$, which is not $t$-analytic.

\medskip

Note that, if questions Q2 and Q3 have a positive answer, then Alling's conjecture is true.
\medskip

Finally, we mention that a  class of closed prime ideals given by Su\'arez \cite{su}
is  a special case of our class whose  construction can be done  using $t$-analytic sets.
In fact,  fix $y\in G$. Consider the following system of $t$-analytic sets:
$$\mathfrak S:=\{\ov{P(x)}:  y\in \ov{P(x)}\}$$
 and let $E=\Union_{M\in \mathfrak S} M$.
  Su\'arez \cite{su} showed that
the ideal $I_{\H}(E)$ is a closed prime ideal. But  $\mathfrak S$ is a chain (see \cite{su});  so we
have that   $E$ actually is $t$-analytic. So  his result is a consequence of our Theorem
\ref{main}.
\bigskip

{\bf Acknowledgements}\medskip

Part of this paper was created during  a visit of the first author at the Universit\'e Paul Verlaine,
Metz, in 2004. He  thanks the LMAM for its generous support.
The authors would also like to thank Norbert Steinmetz and T.W. Gamelin
for confirming that the Riemann map
in Example \ref{comb} does admit a continuous extension as claimed.

\end{document}